\documentclass[a4paper,draft]{amsproc}
\usepackage{amssymb}
\usepackage{amscd}
\usepackage{mathrsfs}
\usepackage{color,xcolor}

\theoremstyle{plain}
 
 \newtheorem{prop}{Proposition}[section]
 \newtheorem{lem}{Lemma}[section]
 
\theoremstyle{definition}
 
 \newtheorem{dfn}{Definition}[section]

 \newtheorem*{con2}{Condition A}
 \newtheorem*{the1}{Theorem A}
\newtheorem*{the2}{Theorem B}

\newtheorem*{open}{Open problem}

\theoremstyle{remark}
 \newtheorem{rem}{Remark}[section]
 \numberwithin{equation}{section}

\renewcommand{\leq}{\leqslant}
\renewcommand{\geq}{\geqslant}

\def\e{\varepsilon}
\title[Scale pressure for amenable group actions]{Scale pressure for amenable group actions}



\author[Dandan Cheng]{Dandan Cheng}
\author [Qian Hao]{Qian Hao}
\author[Zhiming Li]{Zhiming Li }

\address{
School of Mathematics,  
Northwest University,   
Xi'an, 710127,
P.R.China}

\email{china-lizhiming@163.com}

\begin{document}

\vspace{18mm} \setcounter{page}{1} \thispagestyle{empty}

\begin{abstract}
In this paper we introduce the notion of scale pressure and measure theoretic scale pressure for amenable group actions. A variational principle for amenable group actions is presented. We also describe these quantities by pseudo-orbits. Moreover, we prove that if $G$ is a finitely generated countable discrete amenable group, then the scale pressure of $G$ coincides with the scale pressure of $G$ with respect to pseudo-orbits.
\end{abstract}

\maketitle
\section{Introduction}
The concept of entropy, introduced into the realm of dynamical systems more than fifty years ago, has become an important ingredient in the characterization of the complexity of dynamical systems. In the late nineties, M. Gromov \cite{gro} proposed a new dynamical
concept of dimension that was meant to extend the usual topological
dimension to broader contexts. This notion, called mean dimension, is defined for continuous maps on compact
metric spaces in terms of the growth rate of refinements of
coverings of the phase space and was shown to be hardly computable
in general. This invariant of topological dynamical systems was used
by E. Lindenstrauss and B. Weiss introduced in \cite{LW} (it is a
very interesting work to study which other properties of the entropy
are maintained by the mean dimension and the metric mean dimension)
to answer a long-standing question in topological dynamics: does
every minimal topological dynamical system embed into $([0,
1]^\mathbb{Z}, \sigma)$? The answer is negative, since any system
embeddable in $([0, 1]^\mathbb{Z}, \sigma)$ has topological mean
dimension at most one and \cite{LW} constructed a minimal system
with a mean dimension strictly greater than one. This concept,
inspired by the topological entropy, turns out to be a metric
version of Gromov's notion. Moreover, it is an upper bound for the
mean dimension and, in general, easier to estimate. However, it
depends on the metric used, while the topological entropy is metric
independent. The upper and
lower metric mean dimensions, unlike Gromov's concept, depend on
the metric adopted on the space and are nonzero only if the
topological entropy of the dynamics is infinite. Since continuous systems on manifold with infinite topological entropy are generic \cite{Yan}, the metric mean dimension exhibits several intrinsic
features which makes it a rather compelling notion to be studied. There are many elegant approaches to mean dimensions \cite{LT,LT1,V} and mean dimensions with potential \cite{MT}.

This work was intended as an attempt to give a unified way to analyze systems with infinite entropy by introducing a notion of pressure-like complexity functions, which is called \emph{scale pressure}, with respect to a scale function $s$ and a potential function $f$. A variational principle of scale pressure is presented. We also describe these quantities by pseudo-orbits.
\section{Scale Pressures}

More precisely, suppose that $X$ is a compact metric space with metric $d$, $\mathcal{B}$ is the Borel $\sigma$-algebra, $G$ is a countable discrete amenable group and $C(X,\mathbb{R})$ is the Banach algebra of real-valued continuous functions on $X$ equipped with the supreme norm: $\|\varphi\|:=\max\limits_{x\in X}|\varphi(x)|$, $\varphi\in C(X,\mathbb{R})$. Let $(X,G)$ be a $G$-action topological dynamical system, $\mathcal{M}_G(X)$ be collection of the $G$-invariant Borel probability measures of $X$ and $\mathcal{E}_G(X)$ be collection of the $G$-ergodic Borel probability measures of $X$. Recall that a group $G$ is said to be \emph{amenable} if there exists a sequence of finite subsets $\{F_n\}_{n\in \mathbb{N}}$ of $G$ which are asymptotically invariant, i.e.,$$\lim_{n\rightarrow+\infty}\frac{|F_n\vartriangle gF_n|}{|F_n|}=0, \text{ for all } g\in G.$$
Such sequences $\{F_n\}_{n\in \mathbb{N}}$ are called F{\o}lner sequences. We denote the collection of nonempty finite subsets of $G$ by $F(G)$. For each $F\in F(G)$, set metrics $d_F(x,y)=\max\limits_{g\in F}d(gx,gy)$.

Recall that a F{\o}lner sequence $\{F_n\}_{n\in \mathbb{N}}$ in $G$ is \emph{tempered} if there exists a constant $C$ which is independent of $n$ such that $$|\bigcup_{k<n}F_k^{-1}F_n|\leq C|F_n|,\text{ for all }  n\in\mathbb{N}.$$

\begin{dfn}
Let $\e>0$, $E,J,K\subseteq X$ and $F\in F(G)$. If for any $x\in K$ there exists $y\in J$ such that $d_F(x,y)< \e$,\ then we call $J$ is an\emph{ $(F,\e)$-spanning set of $K$}.\ If for any $x,y\in E$ we have $d_F(x,y)>\varepsilon$,\ then we call $E$ is an\emph{ $(F,\e)$-separated set}.
\end{dfn}
Let $sep_{\e,F}(X,G)$ denote the maximal cardinality of an $(F,\e)$-separated subset of $X$ and $spa_{\e,F}(X,G)$ denote the smallest cardinality of an $(F,\e)$-spanning subset of $X$.
\begin{dfn}
A function $s:(0,1)\to(0,\infty)$ is called a \emph{scale function}, if for any $\lambda\in (0,\infty)$, $\lim\limits_{x\rightarrow0}\frac{s(\lambda x)}{s(x)}=1.$ We denote by $\mathcal{S}$ (resp. $\mathcal{S}^*$) the set of all (resp. non-increasing) scale functions.
\end{dfn}
Let $\mathcal{S}^{**}=\{s\in\mathcal{S}^*: \limsup\limits_{\e\rightarrow0}\frac{\e\log\e}{s(\e)}=0\}.$

For any F{\o}lner sequence $\{F_n\}_{n\in \mathbb{N}}$, let $S_{F_n,\varphi}(x)=\sum\limits_{g\in F_n}\varphi(gx).$
For any\ $\varepsilon>0$,\ $s\in\mathcal{S}$, we put

\begin{eqnarray*}
&&Q_{\e,F_n}(\varphi,s):=\\
&&\inf\{\sum_{x\in J}\exp (s(\e)S_{F_n,\varphi}(x)):\ J\ is\ an \ (F_n,\varepsilon)\ spanning\ set\ of\ X\};
\end{eqnarray*}
\begin{eqnarray*}
&&P_{\e,F_n}(\varphi,s):=\\
&&\sup\{\sum_{x\in E}\exp (s(\e)S_{F_n,\varphi}(x)):\ E\ is\ an \ (F_n,\varepsilon)\ separated\ set\ of\ X\}.
\end{eqnarray*}

 Similarly, let
  \begin{eqnarray*}
  &&q_{ \e,F_n}(\varphi,s):=\\ &&\inf\{\sum_{A\in \alpha}\inf_{x\in A}\exp (s(\e)S_{F_n,\varphi}(x)):\ \alpha\ \mbox{is\ a\ finite\ cover\ with\ mesh}(\alpha,d_{F_n})<\e\};
  \end{eqnarray*}
\begin{eqnarray*}
&&p_{\e,F_n}(\varphi,s):=\\
&&\inf\{\sum_{A\in \alpha}\sup_{x\in A}\exp (s(\e)S_{F_n,\varphi}(x)):\ \alpha\ \mbox{is\ a\ finite\ cover\ with\ mesh}(\alpha,d_{F_n})<\e\},
\end{eqnarray*}
where mesh$(\alpha,d_{F_n}):=\max\limits_{A\in\alpha}\mbox{diam}_{d_{F_n}}(A)$.
\begin{rem}\label{re1}
\begin{item}
\item (1) $0\leq Q_{ \e,F_n}(\varphi,s)\leq\|\exp(s(\e)S_{F_n,\varphi})\|spa_{\e,F_n}(X,G)$, \\\ \ \,\,\,$0\leq P_{\e,F_n}(\varphi,s)\leq\|\exp(s(\e)S_{F_n,\varphi})\|sep_{\e,F_n}(X,G)$.\\
\item (2) $P_{\e,F_n}(0,s)=sep_{\e,F_n}(X,G)$, $Q_{\e,F_n}(0,s)=spa_{\e,F_n}(X,G)$.\\
\item (3) $Q_{\e,F_n}(\varphi,s)\leq P_{\e,F_n}(\varphi,s)\leq p_{\e,F_n}(\varphi,s)$.\\
\item (4) If $d(x,y)<\e$ implies $|\varphi(x)-\varphi(y)|\leq\delta$, then $$p_{\e,F_n}(\varphi,s)\leq
    e^{|F_n|\delta s(\e)}q_{\e,F_n}(\varphi,s).$$\\
\end{item}
\end{rem}




For $\varphi\in C(X,\mathbb{R})$, $s\in\mathcal{S}$ and $\e>0$, put
$$Q_{\e}(\varphi,s):=
\limsup_{n\rightarrow\infty}\frac{1}{|F_n|}\log Q_{\e,F_n}(\varphi,s);$$
$$P_{\e}(\varphi,s):=
\limsup_{n\rightarrow\infty}\frac{1}{|F_n|}\log P_{\e,F_n}(\varphi,s);$$

$$p_{\e}(\varphi,s):=\limsup_{n\rightarrow\infty}\frac{1}{|F_n|}\log p_{\e,F_n}(\varphi,s);$$

$$q_{\e}(\varphi,s):=\limsup_{n\rightarrow\infty}\frac{1}{|F_n|}\log q_{\e,F_n}(\varphi,s).$$


Let
$$Q(\varphi,s):=\limsup_{\e\rightarrow0}\frac{Q_{\e}(\varphi,s)}
{s(\e)}; $$
$$P(\varphi,s):=
\limsup_{\e\rightarrow0}
\frac{P_{\e}(\varphi,s)}{s(\e)};$$
$$ p(\varphi,s):=
\limsup_{\e\rightarrow0}
\frac{p_{\e}(\varphi,s)}{s(\e)};$$
$$ q(\varphi,s):=
\limsup_{\e\rightarrow0}
\frac{q_{\e}(\varphi,s)}{s(\e)}.$$

The following Proposition is a direct consequence of Remark \ref{re1} (3) and (4).

\begin{prop}\label{p12}For $s\in\mathcal{S}$, and $\varphi\in C(X,\mathbb{R})$,
$$ Q(\varphi,s)\leq  P(\varphi,s)\leq  p(\varphi,s)\leq  q(\varphi,s).$$
\end{prop}
Denote by $C(X,\mathbb{R}^+)$ the set of nonnegative continuous functions.
\begin{prop}\label{p1} For $\varphi\in C(X,\mathbb{R}^+)$, $s\in\mathcal{S}^*$,
$$Q(\varphi,s)=P(\varphi,s)=p(\varphi,s)=q(\varphi,s).$$

\end{prop}
\begin{proof}
By Proposition \ref{p12} and Remark 1.1 (4), we are left to show that $Q(\varphi,s)\geq q(\varphi,s)$.
 For any $\e>0$ and $(F_n,\frac{\e}{2})$ spanning set $J$, then the $d_{F_n}$ balls centered at $J$ of radius $\frac{\e}{2}$ cover $X$ and these balls form an open cover $\alpha_J$ with mesh$(\alpha,d_{F_n})<\e$.

\begin{eqnarray*}
&&\limsup_{\e\to0}\frac{\limsup\limits_{n\to\infty}\frac{1}{|F_n|}\log\sum\limits_{y\in J}\exp (s(\frac{\e}{2})S_{F_n,\varphi}(y))}{s(\frac{\e}{2})}\\
&\geq& \limsup_{\e\to0}\frac{\limsup\limits_{n\to\infty}\frac{1}{|F_n|}\log\sum\limits_{A\in \alpha}\inf\limits_{x\in A}\exp (s(\e)S_{F_n,\varphi}(x))}{s(\frac{\e}{2})}.
\end{eqnarray*}
In the above inequality, we use the fact that $\varphi\in C(X,\mathbb{R}^+)$ and $s\in\mathcal{S}^*$.
 Therefore, $Q(\varphi,s)\geq  q(\varphi,s)$.


\end{proof}

\begin{dfn}
The map $$SP(\cdot,\cdot): C(X,\mathbb{R})\times \mathcal{S}\rightarrow \mathbb{R}\cup\{\infty\}, \varphi\mapsto SP(\varphi,s)=Q(\varphi,s)$$ is called \emph{scale pressure of} $G$ with respect to $s$ and $\varphi$.
\end{dfn}
For any invariant measure $\mu\in \mathcal{M}_G(X)$, $\varepsilon>0$, $0<\delta<1$, $s\in\mathcal{S}$\ we put

\begin{eqnarray*}
&&P_{\mu,\delta,\e,F_n}(\varphi,s):=\\
&&\inf\{\sum_{x\in E}\exp (s(\e)S_{F_n,\varphi}(x)):\ E \mbox{ is\ an}\ (F_n,\varepsilon)\mbox{\ spanning\ set\ of}\ D\subset X\ \mbox{with}\ \mu(D)\geq 1-\delta\};
\end{eqnarray*}
Let
$$ P_{\mu,\delta,\e}(\varphi,s):=
\limsup_{n\rightarrow\infty}\frac{1}{|F_n|}\log P_{\mu,\delta,\e,F_n}(\varphi,s),$$
and $$P_{\mu,\delta}(\varphi,s):=\limsup_{\e\to 0}\frac{P_{\mu,\delta,\e}(\varphi,s)}{s(\e)}.$$

For any finite measurable partition $\xi$ of $(X,\mathscr{B})$ and any $\mu\in \mathcal{M}_G(X)$, write
$h_{\mu}(G,\xi):=\lim\limits_{n\rightarrow\infty}-\frac{1}{|F_n|}\sum\limits_{A\in\xi_{F_n}}\mu(A)\log\mu(A)$,
where $\xi_{F_n}=\bigvee\limits_{g\in F_n}g^{-1}\xi$.
\begin{dfn}\label{entropy}Let $(X,G)$ be a $G$-action topological dynamical system, $\mu\in \mathcal{M}_G(X)$, we call
$$h_{\mu}(G):= \sup\{h_{\mu}(G,\xi):\ \xi\ \mbox{is\ a\ finite\ measurable\ partition\ of\ }(X,\mathscr{B})\}$$ the \emph {measure-theoretic entropy} of $G$ with respect to $\mu$.
\end{dfn}
For $\mu\in \mathcal{M}_G(X)$ define $N_{\mu}(F_n,\e,\delta)$ as the minimal number of dynamical balls $B_{F_n}(x,\e)$ needed to cover a set of measure strictly bigger than $1-\delta$. Then define $h_{\mu,\delta,\e}(X,G)=\limsup\limits_{n\to\infty} \frac{1}{|F_n|}\log N_{\mu}(F_n,\e,\delta)$. It was proven by Y. Zhao in \cite{Y} that $h_\mu(G)=\lim\limits_{\e\to 0} h_{\mu,\delta,\e}(X,G)$ for any $\mu\in \mathcal{E}_T(G)$, $\delta\in (0,1)$ and tempered F{\o}lner sequence $\{F_n\}_{n\in \mathbb{N}}$ in $G$ with $\lim\limits_{n\rightarrow\infty}\frac{|F_n|}{\log n}=\infty$.
\begin{prop}\label{pp2}
Let $\{F_n\}_{n\in \mathbb{N}}$ be a tempered F{\o}lner sequence in $G$ with $\lim\limits_{n\rightarrow\infty}\frac{|F_n|}{\log n}=\infty$, $s\in\mathcal{S}$ and $\e\in(0,1)$.  For $\varphi\in C(X,\mathbb{R})$, $\mu\in \mathcal{E}_G(X)$,
$$\lim_{\delta \rightarrow 0}P_{\mu,\delta}(\varphi,s)=\lim_{\delta \rightarrow 0}\limsup_{\e\to 0}\frac{h_{\mu,\delta,\e}(X,G)}{s(\e)}+\int \varphi d\mu.$$
\end{prop}
\begin{proof}
The proof is further divided into two steps.

\emph{Step 1} First, we show that\begin{equation}\label{equ1}
\lim_{\delta \rightarrow 0}P_{\mu,\delta}(\varphi,s)\leq\lim_{\delta \rightarrow 0}\limsup_{\e\to 0}\frac{h_{\mu,\delta,\e}(X,G)}{s(\e)}+\int \varphi d\mu.
\end{equation}
For any $\tau>0$, there is $0<\e<\tau$ such that if $x,y\in X$, $d(x,y)<\e$, then $|\varphi(x)-\varphi(y)|<\tau.$
Since $\mu$ is ergodic, by Birkhoff ergodic theorem and Egorov theorem, there is a measurable set $B\in\mathscr{B}$ with $\mu(B)>1-\frac{\delta}{2}$ satisfies the following:
For every $i\in\mathbb{N}$ there exists $N_i\in\mathbb{N}$ such that $N_{i+1}>N_i$ and for any $n\geq N_i$, $x\in B$, $|\frac{1}{|F_n|}s(\e)S_{F_n,\varphi}(x)-\int s(\e)\varphi d\mu|<\frac{1}{i}.$

We choose a sequence of positive integers $\{n_i\}_{i\in\mathbb{N}}$ satisfying:
\begin{item}
\item (1) $n_i\geq N_i$, for any $i\in\mathbb{N}$;
\item (2) $P_{\mu,\delta,\e}(\varphi,s)=\lim\limits_{i\rightarrow\infty}\frac{1}{|F_{n_i}|}\log P_{\mu,\delta,\e,F_{n_i}}(\varphi,s)$.
\end{item}

Let $C_{n_i}$ be a set with $\mu(C_{n_i})>1-\frac{\delta}{2}$ and
$D_{n_i}$ be an $(F_{n_i},\e)$ spanning set of $C_{n_i}$ with card$(D_{n_i})=N_\mu(F_{n_i},\e,\frac{\delta}{2})$.

Put $E_{n_i}=B\cap C_{n_i}$, then $\mu(E_{n_i})>1-\delta.$ Let $J_{n_i}\subset D_{n_i}$ be an $(F_{n_i},\e)$ spanning set of $E_{n_i}$ with smallest cardinality.
Therefore,
\begin{eqnarray*}
&&\sum_{x\in J_{n_i}}\exp (s(\e)S_{F_{n_i},\varphi}(x))\\
&\leq&\sum_{x\in J_{n_i}}\exp [s(\e)|F_{n_i}|(\int \varphi d\mu+\frac{1}{i})]\\
&\leq&card (D_{n_i})\exp [s(\e)|F_{n_i}|(\int \varphi d\mu+\frac{1}{i})].
\end{eqnarray*}
Therefore $$ P_{\mu,\delta,\e,F_{n_i}}(\varphi,s)\leq card (D_{n_i})\exp[|F_{n_i}|(s(\e)\int \varphi d\mu+s(\e)\frac{1}{i})].$$
Hence, $$P_{\mu,\delta,\e}(\varphi,s)\leq h_{\mu,\frac{\delta}{2},\e}(X,G)+s(\e)\int \varphi d\mu.$$

Dividing by $s(\e)$ and letting $\e\rightarrow 0$, we have (\ref{equ1}).

\emph{Step 2} Next, we show that \begin{equation}\label{equ2}
\lim_{\delta \rightarrow 0}P_{\mu,\delta}(\varphi,s)\geq\lim_{\delta \rightarrow 0}\limsup_{\e\to 0}\frac{h_{\mu,\delta,\e}(X,G)}{s(\e)}+\int \varphi d\mu.
\end{equation}
By Birkhoff ergodic theorem and Egorov theorem, similarly to \emph{Step 1} we can find a measurable set $B\in\mathscr{B}$ with $\mu(B)>1-\delta$ and a sequence of positive integers $\{n_i\}_{i\in\mathbb{N}}$ satisfying the following:

For every $i\in\mathbb{N}$ there exists $N_i\in\mathbb{N}$ such that for any $i\in \mathbb{N}$, $x\in B$, $|\frac{1}{|F_{n_i}|}s(\e)S_{F_{n_i},\varphi}(x)-\int s(\e)\varphi d\mu|<\frac{1}{i},$ and
$$h_{\mu,2\delta,\e}(X,G)=\lim\limits_{i\rightarrow\infty}\frac{1}{|F_{n_i}|}\log N_{\mu,2\delta,\e,F_{n_i}}(X,G).$$
Let $C_{n_i}$ be a set with $\mu(C_{n_i})>1-\delta$ and
$D_{n_i}$ be an $(F_{n_i},\e)$ spanning set of $C_{n_i}$. Put $E_{n_i}=B\cap C_{n_i}$, then $\mu(E_{n_i})>1-2\delta.$ Let $J_{n_i}\subset D_{n_i}$ be an $(F_{n_i},\e)$ spanning set of $E_{n_i}$ with smallest cardinality.
Therefore,
\begin{eqnarray*}
&&\sum_{x\in D_{n_i}}\exp (s(\e)S_{F_{n_i},\varphi}(x))\\
&\geq&\sum_{x\in J_{n_i}}\exp (s(\e)S_{F_{n_i},\varphi}(x))\\
&\geq&\sum_{x\in J_{n_i}}\exp [s(\e)|F_{n_i}|(\int \varphi d\mu-\frac{1}{i})]\\
&=&card (J_{n_i})\exp [s(\e)|F_{n_i}|(\int \varphi d\mu-\frac{1}{i})]\\
&\geq&N_\mu(F_{n_i},\e,2\delta)\exp [s(\e)|F_{n_i}|(\int \varphi d\mu-\frac{1}{i})].
\end{eqnarray*}
Hence,
\begin{eqnarray*}
&&P_{\mu,\delta,\e}(\varphi,s)\\
&\geq&h_{\mu,2\delta,\e}(X,G)+s(\e)\int \varphi d\mu.
\end{eqnarray*}

Dividing by $s(\e)$ and letting $\e\rightarrow 0$, we have (\ref{equ2}).

\end{proof}
\begin{dfn}
For any $0<\delta<1$,  $\mu\in \mathcal{M}_G(X)$, $s\in\mathcal{S}$, the map $$SP_{\mu}(\cdot,\cdot): C(X,\mathbb{R})\times\mathcal{S}\rightarrow \mathbb{R}\cup\{\infty\}, \varphi\mapsto SP_{\mu,\delta}(\varphi,s):=P_{\mu,\delta}(\varphi,s)$$ is called \emph{measure-theoretic scale pressure of} $G$ with respect to $s$ and $\varphi$.
\end{dfn}
One of our main results is the following variational principle, we need a condition.
For any finite measurable partition $\xi$ and $r>0$, let $U_{r}(A):=\{x\in A:\exists y\in A^c,\ \mbox{with}\ d(x,y)<r\}$ and $U_{r}(\xi):=\bigcup\limits_{A\in\xi}U_{r}(A)$. Since $\bigcap\limits_{r>0}U_{r}(\xi)=\partial\xi$, $\lim\limits_{r\rightarrow 0}\mu(U_{r}(\xi))=\mu(\partial\xi)$, for any $\mu\in\mathcal{M}_G(X)$, where $\partial\xi:=\bigcup\limits_{A\in\xi}\partial A$ and $\partial A$ is the boundary of $A$.

If a finite measurable partition $\xi$ satisfies $\mu(\partial\xi)=0$ for some $\mu\in\mathcal{E}_G(X)$, then for any $\gamma>0$, we can find $0<r<\gamma$ such that $\mu(U_{r}(\xi))<\gamma.$ Let $r_{\mu,\gamma}:=\sup\{r\in \mathbb{R}^+: \exists\,\mbox{finite\,measurable\,partition}\,\xi
\,\mbox{with}\,\mu(\partial\xi)=0,diam(\xi)<\gamma\,\mbox{and}\,\mu(U_{r}(\xi))<\gamma\}$
and $r_{\gamma}:=\inf\limits_{\mu\in\mathcal{E}_G(X)}r_{\mu,\gamma}$.

\begin{con2}For any $\gamma>0$, $r_{\gamma}>0$ and
$\limsup\limits_{\gamma\rightarrow0}\frac{s(r_{\gamma})}{s(\gamma)}=1.$
\end{con2}
\begin{rem}
If we choose $s\equiv 1$, then Condition A holds for the system such that for any $\gamma>0$, $r_{\gamma}>0$. If we take $s(x)=-\log x$, for example, a trivial example satisfying Condition A is one dimensional uniquely ergodic systems whose ergodic measure is Lebesgue measure.
\end{rem}
For any $\e\in(0,1)$, $s\in\mathcal{S}^{**}$, $\varphi\in C(X,\mathbb{R})$ and $0<\delta<1$, let $$sp(\varphi,s,\delta):=\limsup\limits_{\e\to 0}\frac{\sup\limits_{\mu\in \mathcal{M}_G(X)}P_{\mu,\delta,\e}(\varphi,s)}{s(\e)}.$$
\begin{the1}\label{A} Let $(X,G)$ be a $G$-action topological dynamical system satisfying Condition A. $\{F_n\}_{n\in\mathbb{N}}$ is a tempered F{\o}lner sequence in $G$ with $\lim\limits_{n\rightarrow\infty}\frac{|F_n|}{\log n}=\infty$. For any $S\in\mathcal{S}^{**}$, $\varphi\in C(X,\mathbb{R})$ and $0<\delta<1$,
 $$SP(\varphi,s)=sp(\varphi,s,\delta).$$
\end{the1}
\begin{rem}
The condition $\lim\limits_{n\rightarrow\infty}\frac{|F_n|}{\log n}=\infty$ on the tempered F{\o}lner sequence $\{F_n\}_{n\in\mathbb{N}}$ is only to guarantee the application of Shannon-McMillan-Breiman theorem (see for example \cite{L2}). Therefore, neither the quantity $P_{\mu,\delta,\e}(\varphi)$ nor the result depends on the choice of F{\o}lner sequence.
\end{rem}
To prove theorem A, we need the following lemma.
\begin{lem}\label{bianfen}
For any $\gamma>0,$ there exists $\mu_{\gamma}\in \mathcal{M}_G(X)$ such that for all finite measurable partition $\xi$ with diam$(\xi)<\gamma$ and $\mu_{\gamma}(\partial\xi)=0$, we have $P_{\gamma}(\varphi,s)\leq h_{\mu_{\gamma}}(G,\xi)+S(\gamma)\int \varphi d\mu_{\gamma}$. Moreover,  $\mu_{\gamma}$ can be chosen to be ergodic.
\end{lem}
\begin{proof}

The proof follows the same line of \cite[Theorem 1.2]{D}, with slight modification. For any $\gamma>0$, any tempered F{\o}lner sequence $\{F_n\}_{n\in\mathbb{N}}$ with $\lim\limits_{n\to \infty}\dfrac{|F_n|}{\log n}=\infty$, let $E_{n}=\{x_1,...,x_{sep_{\gamma,F_n}(X,G)}\}$ be a maximal collection of $(F_n,\gamma)$-separated points in $X$ with $$\log\sum\limits_{x\in E_{n}}\exp (s(\gamma)S_{F_n,\varphi}(x))\geq \log P_{\gamma,F_n}(\varphi,s)-1.$$ Define $$\sigma_n=\dfrac{\sum\limits_{y\in E_{n}}\exp (s(\gamma)S_{F_n,\varphi}(y))\delta_y}{\sum\limits_{x\in E_{n}}\exp (s(\gamma)S_{F_n,\varphi}(x))},$$
where $\delta_x$ is the probability measure supported at $x$. Then define
$$\mu_n=\dfrac{1}{|F_n|}\sum_{g\in {F_n}} \sigma_n\circ g^{-1}.$$ 
Consider a subsequence $\{n_k\}_{k\in \mathbb{N}}$ such that $$P_{\gamma}(\varphi,s)=\lim_{k\to \infty}\dfrac{1}{|F_{n_k}|}\log P_{\gamma,n_k}(\varphi,s)$$ and $\{\mu_{n_k}\}_{k\in \mathbb{N}}$ converges to some $\mu_{\gamma}\in \mathcal{M}_G(X)$.
Let $\xi$ be a finite measurable partition with diam$(\xi)<\gamma$ and $\mu(\partial\xi)=0$. Recall that $\xi_{F_n}=\bigvee\limits_{g\in {F_n}}g^{-1}\xi$.
Since each element of $\xi_{F_n}$ contains at most one element of $E_n$, by the definition of $\sigma_n$, we have
\begin{eqnarray*}
&&H_{\sigma_n}(\xi_{F_n})+s(\gamma)\int S_{F_n,\varphi}(x) d\sigma_n\\
&=&\sum_{y\in E_{n}}\sigma_n(\{y\})[s(\gamma)S_{F_n,\varphi}(y)-\log\sigma_n(\{y\})]\\
&=&\log\sum_{y\in E_{n}}\exp (s(\gamma)S_{F_n,\varphi}(y)),
\end{eqnarray*}
where $H_{\sigma_n}(\xi_{F_n}):=\sum\limits_{A\in\xi_{F_n}}-\sigma_n(A)\log\sigma_n(A)$.
By a similar argument of \cite[Lemma 3.1(3)]{W}, for any finite subset $F$ of $G$, we have
$$H_{\sigma_n}(\xi_{F_n})\leq \frac{1}{|F|}\sum_{g\in F_n} H_{\sigma_n\circ g^{-1}}(\xi_F)+\frac{1}{|F_n|}|F^{-1}F_n\backslash F_n|\log \mbox{card}(\xi),$$
where $\mbox{card}(\xi)$ denotes the number of cells of partition $\xi$. Hence,
\begin{eqnarray*}
&&\frac{1}{|F_n|}H_{\sigma_n}(\xi_{F_n})\\
&\leq&\frac{1}{|F|}\frac{1}{|F_n|}\sum_{g\in F_n}H_{\sigma_n\circ g^{-1}}(\xi_F)+\frac{|F^{-1}F_n\backslash F_n|}{|F_n|}\log\mbox{card}(\xi)\\
&\leq&\frac{1}{|F|}H_{\mu_n}(\xi_F)+\frac{|F^{-1}F_n\backslash F_n|}{|F_n|}\log\mbox{card}(\xi).
\end{eqnarray*}
Since $\mu(\partial\xi_F)=0$, we have
\begin{eqnarray*}
&&\limsup_{n\rightarrow\infty}\frac{1}{|F_n|}\log\sum_{y\in E_{n}}\exp (s(\gamma)S_{F_n,\varphi}(y))\\
&=&\lim_{j\rightarrow\infty}\frac{1}{|F_{n_j}|}\log\sum_{y\in E_{n_j}}\exp (s(\gamma)S_{F_{n_j},\varphi}(y))\\
&=&\lim_{j\rightarrow\infty}\frac{1}{|F_{n_j}|}(H_{\sigma_{n_j}}(\xi_{F_{n_j}})+s(\gamma)\int S_{F_{n_j},\varphi}(x) d\sigma_{n_j})\\
&\leq&\lim_{j\rightarrow\infty}(\frac{1}{|F|}H_{\mu_{n_j}}(\xi_F)+\frac{|F^{-1}F_{n_j}\backslash F_{n_j}|}{|F_{n_j}|}\log\mbox{card}(\xi)+\lim_{j\rightarrow\infty}s(\gamma)\int\varphi d\mu_{n_j}\\
&=&\frac{1}{|F|}H_{\mu_\gamma}(\xi_F)+s(\gamma)\int\varphi d\mu_{\gamma}.
\end{eqnarray*}
Taking infimum over all finite subsets of $G$, by \cite[Theorem 4.2]{Td}, $$\limsup\limits_{n\rightarrow\infty}\frac{1}{|F_n|}\log\sum\limits_{y\in E_{n}}\exp (s(\gamma)S_{F_n,\varphi}(y))\leq h_{\mu_\gamma}(G,\xi)+s(\gamma)\int\varphi d\mu_\gamma.$$
We can get that $P_{\gamma}(\varphi,s)\leq h_{\mu_{\gamma}}(G,\xi)+s(\gamma)\int \varphi d \mu_{\gamma}.$

By ergodic decomposition theorem \cite{W}, we can choose $\mu_{\gamma}\in \mathcal{E}_G(X)$ satisfying the desired properties.
\end{proof}
\begin{proof}[proof of theorem A]
The inequality $SP(\varphi,s)\geq sp(\varphi,s,\delta)$ follows directly from definitions.

For any $\tau>0$, there is $0<\e<\min\{\tau,\frac{1-\delta}{2}\}$ such that if $x,y\in X$, $d(x,y)<\e$, then $|\varphi(x)-\varphi(y)|<\tau.$

For any finite measurable partition $\xi$ with $\mu(\partial\xi)=0$ for some $\mu\in\mathcal{E}_G(X)$, we can find $0<r_{\mu}<\e$ such that $\mu(U_{r}(\xi))<\e.$ Suppose that $N_\mu(F_n,r,\delta)$ is the minimal number of $B_{F_n}(x,r)$ balls that cover a set of measure more than $1-\delta$. According to the proof of \cite[Theorem 3.1]{D}, by a combinational arguments on the numbers of the cells $\xi_{F_n}$ and the bowen balls $B_{F_n}(x,r)$, there is $N_1\in\mathbb{N}$ such that for any $n>N_1$, $$N_\mu(F_n,r,\delta)\geq\frac{\exp[|F_n|(h_{\mu}(G,\xi)-\e)]}{D(\frac{\e}{2},N_0,F_n)}
\frac{1-\delta}{4},$$ where $D(\frac{\e}{2},N_0,F_n):=\sum\limits_{m=0}^{[|F_n|\frac{\e}{2}]}(N_0-1)^mC_{|F_n|}^m,$ and $N_0=card (\xi)$. By Stirling formula, $$\lim\limits_{n\rightarrow\infty}\frac{1}{|F_n|}\log D(\frac{\e}{2},N_0,F_n)=\frac{\e}{2}\log(N_0-1)-\frac{\e}{2}\log\frac{\e}{2}-(1-\frac{\e}{2})\log(1-\frac{\e}{2}).$$

Again by Birkhoff ergodic theorem and Egorov theorem, similarly to the proof in \emph{Step 1} of Proposition 1.2, we can find a measurable set $B\in\mathscr{B}$ with $\mu(B)>1-\frac{\delta}{2}$ and a sequence of positive integers $\{n_i\}_{i\in\mathbb{N}}$ satisfies the following:

For every $i\in\mathbb{N}$ there exists $N_i\in\mathbb{N}$ such that for any $i\in \mathbb{N}$, $x\in B$, $|\frac{1}{|F_{n_i}|}s(r)S_{F_{n_i},\varphi}(x)-\int s(r)\varphi d\mu|<\frac{1}{i},$ and
$$P_{\mu,\delta,r}(\varphi,s)=\lim\limits_{i\rightarrow\infty}\frac{1}{|F_{n_i}|}\log P_{\mu,\delta,r,F_{n_i}}(\varphi,s).$$
Let $C_{n_i}$ be a set with $\mu(C_{n_i})>1-\frac{\delta}{2}$ and
$D_{n_i}$ be an $(F_{n_i},r)$ spanning set of $C_{n_i}$. Put $E_{n_i}=B\cap C_{n_i}$, then $\mu(E_{n_i})>1-\delta.$ Let $J_{n_i}\subset D_{n_i}$ be an $(F_{n_i},r)$ spanning set of $E_{n_i}$ with smallest cardinality. Thus for any $x\in J_{n_i}$, there is $\pi(x)\in E_{n_i}$ such that $d_{F_{n_i}}(x,\pi(x))<r$.
Therefore,
\begin{eqnarray*}
&&\sum_{x\in D_{n_i}}\exp (s(\gamma)S_{F_{n_i},\varphi}(x))\\
&\geq&\sum_{x\in J_{n_i}}\exp (s(\gamma)S_{F_{n_i},\varphi}(x))\\
&\geq&\sum_{x\in J_{n_i}}\exp [s(\gamma)|F_{n_i}|(\int \varphi d\mu-\frac{1}{i})]\\
&=&card (J_{n_i})\exp [s(\gamma)|F_{n_i}|(\int \varphi d\mu-\frac{1}{i})]\\
&\geq&N_\mu(F_{n_i},\gamma,\delta)\exp [s(\gamma)|F_{n_i}|(\int \varphi d\mu-\frac{1}{i})]\\
&\geq&
\frac{1-\delta}{4D(\frac{\e}{2},N_0,F_{n_i})}\exp[|F_{n_i}|(h_{\mu}(G,\xi)-\e)]\exp [s(\gamma)|F_{n_i}|(\int \varphi d\mu-\frac{1}{i})]\\
&=&\frac{1-\delta}{4D(\frac{\e}{2},N_0,F_{n_i})}\exp[|F_{n_i}|(h_{\mu}(G,\xi)-\e
+s(\gamma)\int \varphi d\mu-s(\gamma)\frac{1}{i})].
\end{eqnarray*}

Hence, for any $\delta>0$, $\mu\in\mathcal{E}_{G}(X)$ and any finite partition $\xi$ with $\mu(\partial\xi)=0$, we can find $0<r<\e$ such that
\begin{eqnarray*}
&&P_{\mu,\delta,r}(\varphi,s)\\
&\geq&h_{\mu}(G,\xi)+s(r)\int \varphi d\mu+s(r)\tau-\e-[\frac{\e}{2}\log(N_0-1)-\frac{\e}{2}\log\frac{\e}{2}-(1-\frac{\e}{2})\log(1-\frac{\e}{2})].
\end{eqnarray*}

For any $0<\gamma<\e$, by Lemma \ref{bianfen}, we have
\begin{eqnarray*}
&&P_{\gamma}(\varphi,s)\\
&\leq& h_{\mu_{\gamma}}(G,\xi)+s(\gamma)\int \varphi d\mu_{\gamma}\\
&\leq&P_{\mu_{\gamma},\delta,r_{\mu_{\gamma}}(\varphi,s)}+s(r_{\mu_{\gamma}})\tau+\e+[\frac{\e}{2}\log(N_0-1)+\frac{\e}{2}\log\frac{\e}{2}+(1-\frac{\e}{2})\log(1-\frac{\e}{2})]\\
&\leq&\sup_{\mu \in \mathcal{M}_G(X)} P_{\mu,\delta,r_{\mu_{\gamma}}(\varphi,s)}+s(r_{\mu_{\gamma}})\tau+\e+[\frac{\e}{2}\log(N_0-1)+\frac{\e}{2}\log\frac{\e}{2}+(1-\frac{\e}{2})\log(1-\frac{\e}{2})]
\end{eqnarray*}

Dividing by $s(\gamma)$ and letting $\tau\rightarrow 0$. Since $\tau>\e>\gamma>0$, by Condition A, we have $$SP(\varphi,s)\leq sp(\varphi,s).$$
\end{proof}

At last, we pose an open problem, namely, whether we can exchange the order of the limit and supremum in Theorem A.

\begin{open}
Let $(X,G)$ be a $G$-action topological dynamical system, $\varphi\in C(X,\mathbb{R})$, and $0<\delta<1$. Whether
 $$SP(\varphi,s)=\sup_{\mu\in\mathcal{M}_G(X)}SP_{\mu,\delta}(\varphi,s).$$
\end{open}

\section{Pseudo Orbits}  

 Pseudo orbits have proved to be a powerful conceptual tool in dynamical systems. For example, if group action is expansive and has the pseudo orbit tracing property then it is topologically stable \cite{ch};  every expansive action of an amenable group with positive entropy that has the pseudo orbit tracing property must admit off-diagonal asymptotic pairs \cite{tom}. In this section, we investigate scale pressure with respect to pseudo orbits.

In this section, we assume that $G$ is a finitely generated countable discrete amenable group with a \emph{finite} generating set $\mathfrak{G}=\{\mathbf{g}_1,\ldots,\mathbf{g}_m\}$, $m\in\mathbb{N}$.
\begin{dfn}
For $\epsilon>0$ we say that a sequence $\{y_g\}_{g\in G}$ is an \emph{$\epsilon$-pseudo orbit} of an action $G$ (with respect to the generating set $\mathfrak{G}$) if
$d(y_{tg}, t(y_g))<\epsilon, \ \forall t\in\mathfrak{G}, g\in G$.
\end{dfn}
Let $PO_{\epsilon}(\mathfrak{G})$ denote all $\epsilon$ pseudo orbits for $G$, let $X^{G}=\prod\limits_{g\in G}X_g$ with $X_g=X$, and denote the point $\overline{x}$ in $X^{G}$ by $(x_g)_{g\in G}.$
\begin{dfn}
Let $\e>0$, $E,J,K\subseteq X^{G}$ and $F\in F(G)$. If for any $\overline{x}\in K$ there exists $\overline{y}\in J$ such that $d(x_g,y_g)\leq \e$, for all $g\in F$, then we call $J$ is an\emph{ $(F,\e)$ spanning set of $K$}.\ If for any $\overline{x},\overline{y}\in E$ there is a $g\in F$ such that $d(x_g,y_g)>\varepsilon$,\ then we call $E$ is an\emph{ $(F,\e)$ separated set}.
\end{dfn}
Denote by $sep_{\e,F}(X^{G}, K)$ the maximal cardinality of an $(F,\e)$-separated subset of $K\subseteq X^{G}$, by $spa_{\e,F}(X^{G},K)$ the smallest cardinality of an $(F,\e)$-spanning subset of $K\subseteq X^{G}$.
Suppose that $Y\subset PO_\epsilon(\mathfrak{G})$ is a nonempty subset, $\varphi\in C(X,\mathbb{R})$.
For $\varepsilon\in(0,1)$, $s\in\mathcal{S}$ and the F{\o}lner sequence $\{F_n\}_{n\in\mathbb{N}}$, we put
\begin{eqnarray*}
&&POQ_{\e,F_n}(Y,G,\varphi,s):=\\
&&\inf\{\sum_{(x_g)_{g\in G}\in J}\exp (s(\e)\sum_{g\in F_n}\varphi(x_g)):\ J\ is\ an \ (F_n,\varepsilon)\ spanning\ set\ of\ Y\};
\end{eqnarray*}
\begin{eqnarray*}
&&POP_{\e,F_n}(Y,G,\varphi,s):=\\
&&\sup\{\sum_{(x_g)_{g\in G}\in E}\exp (s(\e)\sum_{g\in F_n}\varphi(x_g)):\ E\ is\ an \ (F_n,\varepsilon)\ separated\ set\ of\ Y\}.
\end{eqnarray*}

For $\varphi\in C(X,\mathbb{R})$ $\e\in(0,1)$, $s\in \mathcal{S}$ and the F{\o}lner sequence $\{F_n\}_{n\in\mathbb{N}}$ in $G$, let
$$POP(\varphi,s):=\limsup_{\e\rightarrow0}\frac{\limsup\limits_{n\rightarrow\infty}\frac{1}{|F_n|}\log(\lim\limits_{\epsilon\rightarrow0}POP_{\e,F_n}(PO_\epsilon(\mathfrak{G}),G,\varphi,s))}
{s(\e)}; $$
$$POQ(\varphi,s):=\limsup_{\e\rightarrow0}\frac{\limsup\limits_{n\rightarrow\infty}\frac{1}{|F_n|}\log(\lim\limits_{\epsilon\rightarrow0}POQ_{\e,F_n}(PO_\epsilon(\mathfrak{G}),G,\varphi,s))}
{s(\e)}. $$

\begin{prop}\label{ppp3}For $\varphi\in C(X,\mathbb{R})$ with $\varphi\geq0$, and for $\e\in(0,1)$, $s\in\mathcal{S}^*$,
$POP(\varphi,s)=POQ(\varphi,s)$.
\end{prop}
\begin{proof}
By a similar argument of Proposition \ref{p1}, for any $\delta>0$, $$POQ(\varphi,s)\leq POP(\varphi,s)\leq POQ(\varphi,s)+\delta.$$

\end{proof}
\begin{dfn}
The map $$PSP(\cdot,\cdot): C(X,\mathbb{R})\times \mathcal{S}\rightarrow \mathbb{R}\cup\{\infty\}, \varphi\mapsto PSP(\varphi,s):=POQ(\varphi,s)$$ is called \emph{scale pressure of} $G$ \emph{with respect to pseudo-orbits}.
\end{dfn}
\begin{dfn}
We call $(X,G)$ a \emph{Lipschitz $G$-system} if any $t\in\mathfrak{G}$ is a bi-Lipschitz homeomorphism.
\end{dfn}

For $0\leq\epsilon\leq1$, let $X_\epsilon$ be the closure of $PO_\epsilon$.
Since $G$ is countable, we can arrange $G$ as $G:=\{g_i\}_{i\in\mathbb{N}}$ with $g_0=e$, then $(X_1,\overline{d})$ is a compact metric space where $\overline{d}(\overline{x},\overline{y}):=\sum\limits_{i\in \mathbb{N}}\frac{d(x_{g_i},y_{g_i})}{2^{i}}$, $\overline{x},\overline{y}\in X_{1}$. For any $0\leq\epsilon\leq1$, $X_{\epsilon}$ is a compact subset of $X_1$. Then $G$ act on $X_{\epsilon}$ naturally by $t(x_g)_{g\in G}=(x_{tg})_{g\in G}$. Set $\pi_{\epsilon}:X_{\epsilon}\rightarrow X$, $(x_g)_{g\in G}\mapsto x_e$. It is clear that $\pi_{\epsilon}$ is continuous and $\pi_{\epsilon}=\pi_{1}|_{X_{\epsilon}}$.

\begin{lem}\label{jia}
For any $\eta>0$, there exists $\epsilon_0>0$ such that for any $0\leq\epsilon\leq\epsilon_0$ and $\overline{x}\in X_{\epsilon}$, there exists $\overline{y}\in X_{0}$ with $\overline{d}(\overline{x},\overline{y})<\eta,$ where $X_{0}=\{\overline{x}=(x_g)_{g\in G}\in X^{G}:\ x_{tg}=t(x_g), \ \forall t\in\mathfrak{G}, g\in G\}$.
\end{lem}
\begin{proof}
For any $\eta>0$, there exists $N\in\mathbb{N}$ such that $\sum\limits_{i=N}^{\infty}\frac{D}{2^{i}}<\frac{\eta}{3}$, where $D:=\max\{d(x,y):x,y\in X\}$. Choose $\epsilon_0>0$ so that $d(x,y)<\epsilon_0$ implies $d_N(x,y)<\frac{\eta}{3}$, where $d_N(x,y):=\max\limits_{0\leq i\leq N}d(g_i(x),g_i(y))$. For any $\overline{y}\in X_{\rho}$ where $0\leq\rho\leq\epsilon_0$, $\overline{x}:=(x_g)_{g\in G}$ with $d(x_e,y_e)<\epsilon$ and $x_g=g(x_e)$, $\overline{x}$ will satisfy $\overline{d}(\overline{x},\overline{y})<\eta.$
\end{proof}
In the following Lemmas, we will add $X$ and $G$ to the notations $p$ and $Q$ in order to distinguish different spaces.
\begin{lem}\label{bbb}There exists $\eta_0>0$ such that for any $0<\eta<\eta_0$, $s\in\mathcal{S}$, $F\in F(G)$ and $\varphi\in C(X,\mathbb{R})$,
$$\lim_{\epsilon\rightarrow 0}p_{\eta,F}(X_{\epsilon},G,\varphi\circ\pi_{\epsilon},s)
=p_{\eta,F}(X_{0},G,\varphi\circ\pi_{0},s).$$
\end{lem}
\begin{proof}
For $\eta>0$ and $\epsilon>0$. Let $\alpha_{1,\eta}$ be a finite open cover of $X_1$  with mesh $(\alpha_{1,\eta},\overline{d}_F)<\eta$, and $\alpha_{1,\eta}(\epsilon):=\{A:\ A\in \alpha_{1,\eta},\ A\cap X_{\epsilon}\neq\emptyset\}$.
It is clear that $\alpha_{1,\varepsilon}^0(\epsilon)$ is an open cover of $X_\epsilon$.
For any $\delta>0$, there exists $\eta_0>0$ such that for any $\overline{x},\overline{y}\in X_1$ with $\overline{d}(\overline{x},\overline{y})<\eta_0$, $i\in\mathbb{N}$, $|\varphi\circ\pi_1(g_i(\overline{x}))-\varphi\circ\pi_1(g_i(\overline{y}))|<\frac{\delta}{|F|}$.
By Lemma \ref{jia}, we can choose $\epsilon_0>0$ satisfying that
for any $0\leq\epsilon\leq\epsilon_0$ and $\overline{x}\in X_{\epsilon}$, there exists $\overline{y}\in X_{0}$ with $\overline{d}(\overline{x},\overline{y})<\eta_0.$

So for any $0<\eta<\eta_0$ and $0\leq\epsilon\leq\epsilon_0$, we have
\begin{eqnarray*}
&&\sum_{A\in \alpha_{1,\eta}(\epsilon)}\sup_{x\in A}\exp (s(\eta)S_{F,\varphi\circ\pi_{\epsilon}}(x))\\
&\leq&\sum_{A\in \alpha_{1,\eta}(\epsilon)}\sup_{x\in A}e^{s(\eta)\delta}\exp (s(\eta)S_{F,\varphi\circ\pi_0}(x)).
\end{eqnarray*}

Therefore, $p_{\eta,F}(X_{\epsilon},G,\varphi\circ\pi_{\epsilon},s)
\leq p_{\eta,F}(X_{0},G,\varphi\circ\pi_{0},s)e^{s(\eta)\delta}.$

On the other hand, it is obvious that $p_{\eta,F}(X_{\epsilon},G,\varphi\circ\pi_{\epsilon},s)
\geq p_{\eta,F}(X_{0},G,\varphi\circ\pi_{0},s).$
\end{proof}
\begin{lem}\label{ccc}If $(G,X)$ is a Lipschitz $G$-system, $\varphi\in C(X,\mathbb{R})$ with $\varphi\geq0$, and for $\e\in(0,1)$, $s\in\mathcal{S}^*$, then
$Q(X,G,\varphi,s)=Q(X_{0},G,\varphi,s).$
\end{lem}
\begin{proof}
Since $\pi_0$ is a surjective Lipschitz continuous map with Lipschitz constant $L(\pi_0)\leq 1$ and for any $\e>0$, if $\overline{d}(\overline{x},\overline{y})<\e$, then $d(\pi_0(\overline{x}),\pi_0(\overline{y}))<\e$. Thus if $A$ is an $(F,\e)$ spanning set of $X_0$ then $\pi_0(A)$ is an $(n,\e)$ spanning set of $X$.

Therefore, $$Q(X,G,\varphi,s)\leq Q(X_0,G,\varphi,s).$$

On the other hand, for any $\e>0$, there is $N\in\mathbb{N}$ such that $\sum\limits_{i=N}^{\infty}\frac{D}{2^{i}}<\frac{\e}{3}$, where $D:=\max\{d(x,y):x,y\in X\}$. Choose $\delta=\frac{\e}{4Lip_N}$ so that $d(x,y)<\delta$ implies $d_N(x,y)<\frac{\e}{3}$, where $Lip_N:=\max\limits_{0\leq i\leq N}Lip(g_i)$.
For any $(F,\delta)$ spanning set $A$ of $X$ then $Orb(A):=\{(x_g)_{g\in G}\in X_0: x_e\in A\}$ is an $(F,\e)$ spanning set of $X_0$.

Therefore, $$Q(X,G,\varphi,s)\geq Q(X_0,G,\varphi,s).$$
\end{proof}
\begin{the2}\label{B} Let $(X,G)$ be a Lipschitz $G$-system, if $G$ is a finitely generated countable discrete amenable group, and for $s\in\mathcal{S}^*$, then
$SP(\varphi,s)=PSP(\varphi,s).$
\end{the2}

\begin{proof}
The inequality $SP(\varphi,s)\leq PSP(\varphi,s)$ follows directly from definitions. By Lemma \ref{ccc}, it is sufficient to show that $$SP(\varphi,s)\geq PSP(\varphi,s).$$
For any $0<\eta<\eta_0$ and $0\leq\epsilon\leq\epsilon_0$ as in Lemma \ref{bbb}, if $E$ is an $(F,2\eta)$-separated set, $\alpha$ is an open cover with mesh$(\alpha,\overline{d}_{F_n})<\eta$, then
$$\sum_{(x_g)_{g\in G}\in E}\exp (s(2\eta)\sum_{g\in {F_n}}\varphi(x_g))\leq\sum_{A\in\alpha}\sup_{x\in A}e^{s(\eta)\delta}\exp(s(\eta)S_{F_n,\varphi\circ\pi_\epsilon}(x)).$$
Therefore, $\limsup\limits_{n\rightarrow\infty}\frac{1}{|F_n|}\log(POP_{2\eta,F_n}(PO_\epsilon(\mathfrak{G}),G,\varphi,s))\leq p_\eta(X_\epsilon,G,\varphi\circ\pi_\epsilon,s).$

By Lemma \ref{bbb}, $\lim\limits_{\epsilon\rightarrow0}p_\eta(X_\epsilon,G,\varphi\circ\pi_\epsilon,s)
\leq p_\eta(X_0,G,\varphi\circ\pi_0,s).$

Thus $\limsup\limits_{n\rightarrow\infty}\frac{1}{|F_n|}\log(POP_{2\eta,F_n}(PO_\epsilon(\mathfrak{G}),G,\varphi,s))\leq p_\eta(X_0,G,\varphi\circ\pi_0,s).$

Dividing by $s(\e)$, by Lemma \ref{ccc}, letting $\eta\rightarrow 0$, $$PSP(\varphi,s)\leq SP(\varphi,s).$$
\end{proof}

\bibliographystyle{amsplain}

\end{document}